\documentclass[12pt]{amsart}
\usepackage[unicode]{hyperref}
\usepackage{a4,charter}

\usepackage{amsmath,amsfonts,amssymb,amsthm,amsxtra,latexsym, amscd,amsthm}
\usepackage[all]{xy}

\newtheorem{theorem}{Theorem}[section]
\newtheorem{corollary}[theorem]{Corollary}
\newtheorem{lemma}[theorem]{Lemma}
\newtheorem{proposition}[theorem]{Proposition}

\theoremstyle{definition}

\newtheorem{definition}[theorem]{Definition}
\newtheorem{remarks}[theorem]{Remarks}
\newtheorem{remark}[theorem]{Remark}

\newtheorem{question}[theorem]{Question}

\def\square{\Box}

\newcommand{\inj}{\hookrightarrow}
\def\id{{\sf id}}
 
 \def\ker{{\sf ker}}
\def\spec {\sf Spec}

\begin{document}
\title[Flatness and projectivity over Hopf subalgebras of Hopf algebras over Dedekind rings]{On the flatness and the projectivity over Hopf subalgebras of Hopf algebras over Dedekind rings}
\dedicatory{Herrn Prof. Dr. Bodo Pareigis zum 80. Geburtstag gewidmet}
\author[Nguy\^en Dai Duong]{Nguy\^en Dai Duong}
\email[Nguy\^en Dai Duong]{nguyendaiduongqn@yahoo.com.vn}
\address[Nguy\^en Dai Duong]{Institute of  Mathematics, Vietnam Academy of Science and Technology, Hanoi, Vietnam}

\author[Ph\`ung H\^o Hai]{Ph\`ung H\^o Hai}
\email[Ph\`ung H\^o Hai]{phung@math.ac.vn }
\address[Ph\`ung H\^o Hai]{Institute of  Mathematics, Vietnam Academy of Science and Technology, Hanoi, Vietnam}

\author[Nguy\^en Huy Hung]{Nguy\^en Huy Hung}
\email[Nguy\^en Huy Hung]{hungp1999@yahoo.com}
\address[Nguy\^en Huy Hung]{Hanoi Pedagogical University 2, Vinh Phuc, Vietnam}


\begin{abstract} We study the flatness and the projectivity of Hopf algebras, defined over a Dedekind ring, over their Hopf  subalgebras. We  give a criterion for the faithful flatness and use it to show the faithful flatness of an arbitrary flat Hopf algebra upon its finite normal Hopf subalgebras. For the projectivity over Hopf subalgebras of a Hopf algebra we need some finiteness condition in terms of the module of integrals. In particular we show that the module of integrals is projective of rank one. 
\end{abstract}
\maketitle
\parskip8pt
\section{Introduction} For Hopf algebras defined over a field, a conjecture of Kaplansky 
states that ``a Hopf algebra is free as a module over any Hopf subalgebra''. Although this was quickly
shown to be false in the infinite dimensional case, the finite dimensional case  is true and was proven  by Nichols and Zoeller \cite{NZ89b}. Schneider \cite{Sch93} showed that any Hopf algebra is free over its finite normal  Hopf subalgebras.

This work  is devoted to the study of the same questions for Hopf algebras defined over a Dedekind ring:
 when is a Hopf algebra, defined over a Dedekind ring, faithfully flat/projective over a Hopf subalgebra?
There are already many works devoted to Hopf algebras defined over a ring base. For instance 
Schneider \cite{Sch92} generalized Nichols-Zoeller's result to Hopf algebras over a local ring.
One of our aims is to generalize to the case of 
Hopf algebras over a Dedekind ring Schneider's result: a 
Hopf algebra is projective over a normal Hopf 
subalgebra. We haven't proved the complete 
generalization but still need an extra-condition 
on the finiteness of the Hopf algebra: the 
existence of non-zero integrals.  We show that the module 
of integrals on a projective (over the base Dedekind ring) 
Hopf algebra is projective of rank one (provided it is non-zero). Further we 
show that a projective Hopf algebra possesses a non-zero (left) integral if and only if it is projective as a 
(right) comodule over itself. These results are used to prove the projectivity of a projective Hopf algebra possessing a non-zero integral over a finite normal Hopf subalgebra.

Our approach is to rely on existing results for 
Hopf algebras over fields and to lift information on 
fibers to the global base. This method has been utilized for 
the study of Hopf algebras over rings. For 
instance Pareigis \cite{Par} has used this method 
to show the uniqueness of integrals on finite flat 
Hopf algebras over a Dedekind ring. Our new 
input is a homological lemma (Lemma \ref{flat0}) relating the flatness over the global base with the flatness on the fibers. We also make use of the correspondence between normal  Hopf subalgebras and co-normal quotient Hopf algebras, due originally to Takeuchi and Schneider as well as various equivalences of the categories of modules, comodules and Hopf modules.

The paper is organized as follows. In Section \ref{co-flat} we recall the Takeuchi-Schneider correspondence between normal Hopf subalgebras and co-normal quotients Hopf algebas of a given flat Hopf algebra over a Dedekind ring. This correspondence seems to be well-known by experts but we cannot find any reference suitable to our aim.
In Section \ref{Sect4} we provide the key technical lemma.

\noindent
\begin*{\bf Proposition \ref{flatness}. }\em 
  Let $A$  be an $R$-algebra and $B$ be an $A$-module such that both $A, B$ are flat over $R$.  Then   $B$ is  left faithfully flat over $A$
 if   and only if $B_k$ is left faithfully flat over $A_k$  for $k$ being the fraction field  and any residue field of $R$. 
\end*{}

As a consequence we show that a flat Hopf algebra over a Dedekind ring is flat over any finite saturated normal Hopf subalgebra (Theorem \ref{flatness2}). 

To prove the projectivity over normal Hopf subalgebras we will need some supplementary results, which will be presented in Section \ref{SecInt}. We study Hopf algebras equipped with a non-zero integral. First we consider the modules of integrals on the Hopf algebras in the settings of Theorem \ref{flatness2}.

\noindent\begin*{\bf Theorem \ref{B-C.int}. }\em
Assume that $A$ is an $R$-finite saturated normal Hopf subalgebra of an $R$-flat Hopf algebra $B$. Let $C:=B\otimes_AR$ be the quotient Hopf algebra. Then $B$ possesses a non-zero left (resp. right) integral iff $C$ does.
\end*{}
 
 Then we show that the module of integrals is projective of rank one over the base ring.

\noindent\begin*{\bf Proposition \ref{pro.IntMod}. }\em Let $H$ be an $R$-projective Hopf algebra.
Suppose $H$ possesses a non-zero integral. Then the module of  integrals on $H$ is  an $R$-projective module of rank one. Further the module of integrals respects base change, i.e. $(I^l_H)_k\cong I^l_{H_k}$ for $k$ being any residue field and
$(I^l_{H})_\mathfrak{m}\cong I^l_{H_\mathfrak{m}}$
for any prime ideals $\mathfrak m$.
\end*{}\\[1ex]
The proof is a simple utilization of the theory of rational modules as developed in \cite{Wis} and \cite{CM}.

Next we show that if $H$ possesses a non-zero integral then it is a projective comodule on itself.

\noindent\begin*{\bf Proposition \ref{Integral2}. } 
\em Let $H$ be  an $R$-projective Hopf algebra. Then  $H$ possesses a non-zero integral if and only if $H$ is projective in ${}_{H^*}\!\mathcal{M}$. In particular, if $H$ possesses a non-zero integral then it is projective in $\mathcal M^H$.
\end*{}\\

In Section \ref{Secpro} we show that an $R$-projective Hopf algebra  possessing a non-zero integral is projective over any of its $R$-finite saturated normal Hopf subalgebras.

\noindent\begin*{\bf Theorem \ref{thm.proj}. }\em 
Let $B$ be an $R$-projective Hopf algebra with  a non-zero left integral. Let $A$ be an $R$-finite saturated normal Hopf subalgebra of $B$.  Then $B$ is right  projective over $A$.   \end*{} \\[1ex]
The proof is similar to that of  \cite[Theorem 3.1]{Ps1}, we use several 
equivalences of \mbox{module-,} comodule- and Hopf module categories obtained before. 
This  technique was  utilized in \cite[Section 2]{S.CA}.
\section{Flatness and co-flatness} \label{co-flat}
 In this section we recall the Takeuchi-Schneider correspondence between normal Hopf subalgebras and conormal quotient Hopf algebras.  

Let $R$ be  a Dedekind ring.
In what follows, the tensor product, when not indicated, is understood as the tensor product over $R$.  
We shall frequently make use of the following facts about $R$-modules: a torsion-free module is flat (hence a submodule of a flat module is flat), a finite flat module is projective. For an $R$-module $M$, the torsion submodule is denoted by $M_\tau$, it consists of elements, each annihilated by a non-zero element of $R$. The quotient $M/M_\tau$ is torsion-free. The saturation of a submodule $N$ of $M$ is the preimage in $M$ of the torsion submodule of $M/N$.
$N$ is said to be saturated in $M$ if it is equal to its saturation in $M$.

A coalgebra (resp. Hopf algebra)  over $R$  is  called   $R$-flat (resp. $R$-projective,  $R$-finite) if it is   flat (resp.  projective,   finitely generated) as an $R$-module.

\begin{definition} \label{def_special} Let  $f:A\to B$ be a homomorphism of $R$-flat
 coalgebras.
\begin{enumerate}
\item If $f$ is injective, we shall say that $A$ is a subcoalgebra of $B$ and usually identify $A$ with a subset of $B$ by means of $f$. Notice that the map $f\otimes f:A\otimes A\to B\otimes B$ is injective, thus the coproduct of $A$ is the restriction of that of $B$.
\item $f$  is called saturated if $f(A)$ is saturated as an $R$-submodule in $B$.

\item If $f$ is injective and saturated we shall say that $A$ is a saturated subcoalgebra of $B$. \end{enumerate}
\end{definition}
As $R$ is a Dedekind ring  the image of a (saturated) homomorphism of $R$-flat coalgebras is a (saturated) $R$-flat subcoalgebra.

The following definition is motivated by \cite[Definition 1.1]{Sch93}.
\begin{definition} \label{defn2}  Let $f: A \longrightarrow B$  be a homomorphism of $R$-flat Hopf algebras.
\begin{enumerate}
 \item $f$ is  called normal if  for all $a \in A, b \in B$ we have
$$\sum b_1f(a)S(b_2) \in f(A)\quad \text{and}\quad\sum S(b_1)f(a)b_2  \in f(A).$$ If $f$ is injective then we say that $A$ (by means of $f$) is a normal Hopf subalgebra of $B$.
\item $f$ is called conormal if its kernel ${\sf Ker}(f)$ satisfies: for all  $a \in {\sf Ker}(f)$, we have  $$ \sum a_2 \otimes S(a_1) a_3,
\sum a_2 \otimes a_1 S(a_3) \in{\sf Ker}(f) \otimes A.$$
A Hopf ideal $I$ is called normal if it satisfies the above condition of ${\sf Ker}(f)$, in particular, if $I$ is also saturated then the quotient map $A\to A/I$ is conormal.
\end{enumerate}
\end{definition}

Let $f: A \longrightarrow B$ be a  homomorphism of $R$-flat Hopf algebras.
Set $$A^\text{co$(B)$}:= \{a \in A| \sum a_1 \otimes f(a_2)= a \otimes 1 \},$$ 
$${}^\text{co$(B)$}\!A:= \{a \in A| \sum f(a_1) \otimes a_2= 1 \otimes a \}.$$ 
 These are saturated $R$-submodules of $A$, as they can be characterized as the kernels of homomorphisms.
 
For a Hopf  algebra of $A$, we denote  $A^+:=\ker \varepsilon _A$, the augmentation Hopf ideal of $A$.  The proof of the following lemma is similar to that of  \cite[Lemma 1.3, p. 3342]{Sch93}.
\begin{lemma}\label{normal}   Let $f: A \longrightarrow B$ be a  homomorphism of $R$-flat Hopf algebras.
\begin{enumerate}
\item If $f$ is conormal then   $A^\text{co$(B)$}= {}^\text{co$(B)$}\!A$ is a normal  Hopf subalgebra of  $A$.
\item If $f$ normal then $I:= Bf(A)^+B=Bf(A)^+=f(A)^+B$ is a 
normal Hopf ideal  in $B$.
\end{enumerate}
\end{lemma}
 
For $A\subset B$ a saturated normal Hopf subalgebra, it is not necessary  that the Hopf ideal $A^+B$ is saturated in $B$. However it will be the case if $B$ is faithfully flat as a (left or right)  $A$-module, see \ref{f-co-flat}.

\subsection{}\label{subsection2.1}
Let $C$ be an $R$-flat coalgebra. Then the category $\mathcal M^C$ of right $C$-comodules is abelian. Similarly, the category ${}^C\mathcal M$ of left $C$-comodules is abelian. Let $M\in\mathcal M^C$ and $N\in {}^C\mathcal M$. The co-tensor product $M\square_CN$ is defined as an equalizer:
\begin{equation}\label{eq1}\xymatrix{
0\ar[r]& M\square_CN\ar[r]& M\otimes N\ar[rrr]^{\rho_M\otimes N-M\otimes\rho_N}&&& M\otimes C\otimes N,}
\end{equation}
where $\rho_M$, $\rho_N$ are respectively the coactions of $C$ on $M$ and $N$. We say that $M$ is (right) co-flat over $C$ if the functor $M\square -:{\mathcal M}^C \to \mathcal M_R$ is  exact (this functor is left exact if $M$ is $R$-flat).

The special case of importance is when an $R$-flat coalgebra $B$ is consider as left and right $C$-comodule by means of a coalgebra map $B\to C$. Then for any $M\in\mathcal M^C$, $M\square_CB$ is a right $B$-comodule in  the natural way. In fact, by the $R$-flatness of $B$ we have the following commutative diagram with exact rows:
$$\xymatrix{
0\ar[r]&M\square_CB\ar[r]\ar[d]& M\otimes B
\ar[rr]^{\rho_M\otimes B-M\otimes\lambda_B}\ar[d]_{M\otimes \Delta}&& M\otimes C\otimes B\ar[d]^{M\otimes C\otimes \Delta}\\
0\ar[r]&(M\square_CB)\otimes B\ar[r]&
M\otimes B\otimes B\ar[rr]&&M\otimes C\otimes B\otimes B,}$$
where $\rho_M$ is the coaction of $C$ on $M$ and $\lambda_B$ is the left coaction of $C$ on $B$ given by
$$\xymatrix{\rho_B:B\ar[r]^\Delta& B\otimes B\ar[r]&C\otimes B.}$$
Hence we have a left exact functor
$-\square B:\mathcal M^C\to \mathcal M^B$, called the induction functor. This functor is right adjoint to the restriction functor $\mathcal M^B\to\mathcal M^C$:
\begin{equation}\label{cotensor_adj}{\sf Hom}^C(M,N)\cong {\sf Hom}^B(M,N \square_CB),\quad \text{for $M\in\mathcal M^B$, $N\in \mathcal M^C$.}\end{equation}
The isomorphism above is given by composing a morphism on the left hand side with the map $\rho_M:M\to M\square_CB\subset M\otimes B$:
$$f\mapsto g:=(f\otimes{\sf id})\rho_M.$$
\subsection{}\label{subsection2.2}
Let $A$ be an $R$-flat algebra and $C$ be an $R$-flat coalgebra. A right $(C,A)$-bimodule is an $R$-module $M$ equipped with a left action of $A$, and a right coaction of $C$, say $\rho$, such that
$$\rho(a m)=a\rho(m),\text{ for $a\in A, m\in M$}.$$
Similary, we define left $(C,A)$-bimodules by switching the left and the right (co)actions: $A$ acts on the right and $C$ coacts on the left, by $\lambda$, satisfying $\lambda(m a)=\lambda(m)a$,  for $a\in A, m\in M$. 

Let $M$ be a right $(C,A)$-bimodule. Then for any   left $C$-comodule $N$, there exists a natural action of $A$ on $M\square_CN$ given by the action of $A$ on $M$. Dually, for any right $A$-module $P$, $C$ coacts on $P\otimes_AM$ through the coaction on $M$. 

Assume now that $P$ is right $A$-flat. Then we have a canonical isomorphism
\begin{equation}\label{eq2}P\otimes_A(M\square_CN)\xrightarrow{\sim}
(P\otimes_AM)\square_CN,\end{equation}
obtained by tensoring the exact sequence \eqref{eq1} with $P$ over $A$. We notice that, by the flatness of $P$ over $A$, these two spaces are subspaces of $P\otimes_AM\otimes N$.

Dually, if $N$ is co-flat over $C$, then \eqref{eq2} holds for any $A$-module $P$. Indeed, let $P_1\to P_0\to P\to 0$ be an $A$-free resolution of $P$, then by means of the co-flatness of $N$ over $C$ we have the exactness of the lower sequence in the diagram below forcing the rightmost vertical morphism to be bijective:
$$\xymatrix{
P_1\otimes _A(M\square_CN)\ar[r]\ar[d]_\cong&P_0\otimes _A(M\square_CN)\ar[r]\ar[d]_\cong&P\otimes _A(M\square_CN)\ar[r]\ar[d]&0\\
(P_1\otimes_AM)\square_CN\ar[r]&
(P_0\otimes_AM)\square_CN\ar[r]&
(P\otimes_AM)\square_CN\ar[r]&0}$$

\subsection{}\label{subsection2.3} There is a mirrored version of all claims in \ref{subsection2.2}, in which ``left'' and ``right'' are interchanged. We shall frequently use them later on. For instance, we have
\begin{equation}\label{eq2refl}
(N\square_CM)\otimes_AP\xrightarrow\sim
N\square_C(M\otimes_AP),\end{equation}
if $P$ is left $A$-flat or $N$ is right $C$-co-flat.
\subsection{}\label{subsection2.4} Let $f:A\to B$ be a homomorphism of $R$-flat Hopf algebras. A  Hopf $(B, A)$-module is an $R$-module $M$ equipped with a right $B$-comodule structure and a right $A$-module structure, such that   the comodule structure map $\rho_M: M \longrightarrow M \otimes B$ is $A$-linear, where $A$ acts diagonally on $M\otimes B$. Explicitly we have
$$\rho(ma)=\sum_{m, a} m_0a_1\otimes m_1f(a_2),\quad m\in M, a\in A.$$
 Morphisms of Hopf modules are those maps of the underlying $R$-modules which are both $A$-linear and $B$-colinear. Denote by $\mathcal{M}^B_A$ the category of Hopf $(B, A)$-modules.

We shall be interested in two cases: either $f$ is a normal inclusion of Hopf algebras or $f$ is a conormal quotient map of Hopf algebras.

\subsection{Normal inclusion of Hopf algebras} 
Let $A$ be a normal Hopf 
subalgebra of a flat Hopf algebra $B$. 

Assume that $C:=B/BA^+$ is $R$-flat, i.e. the ideal $A^+B=BA^+$ is saturated in $B$. Then the quotient map $\pi:B\to C$ is 
conormal (Lemma \ref{normal}~(ii)). Further, $B$ is a {\em left} $(C,A)$-bimodule with respect to the natural (co)actions.
For $N \in \mathcal{M}^{C}$, there is a Hopf $(B,A)$-module structure on  $N\square_{C}B $, where the coaction of $B$
is the induced coaction, as in \ref{subsection2.1}, and the action of $A$ is induced from the action on $B$, as in \ref{subsection2.3}.
Thus we have functor
$$\Psi:  \mathcal{M}^{C}\longrightarrow \mathcal{M}^B_A,\quad \Psi (N)=N \square_{C}B.$$
Regarding $R$ as a left $A$-module by means of $\varepsilon_A$, we can construct a functor: $$\Phi:\mathcal{M}^B_A
\longrightarrow \mathcal{M}^{C},\quad 
M\longmapsto  M\otimes_AR.$$
 \begin{lemma}\label{lem.adj1} Assume that $C:=B/BA^+$ is $R$-flat. Then the functor $\Psi$ is right adjoint to the functor
$ \Phi$: 
$${\sf Hom}^C(M\otimes_AR,N)\cong{\sf Hom}^B_A(M,N\square_CB).$$
The adjunctions are induced from the maps $\rho_M:M\to M\square_CB$ and   $\varepsilon_B:N\square_CB\to N$. 
\end{lemma}
\begin{proof}
According to \ref{subsection2.1} we have a functorial isomorphism
$${\sf Hom}^C(M,N)\cong {\sf Hom}^B(M,N\square_CB),\quad M\in \mathcal M^B, N\in\mathcal M^C,$$
given by composing a morphism on the left hand side with the map $\rho_M:M\to M\square_CB$:
$f\mapsto g:=(f\otimes B)\rho_M.$
Thus we see that if $g$ is $A$-linear, i.e., for $a\in A$, $m\in M$ we have
$$\sum f(m_0a_1)\otimes m_1a_2=\sum f(m_0)\otimes m_1a,$$
then, applying $\varepsilon$ on the second tensor factor, we get
$$f(ma)=\varepsilon(a)f(m).$$
That is, $f: M\to N$ factors as the composition of $\bar f:M\otimes_AR\to N$ and the quotient map $q_M:M\to M\otimes_AR$. The converse also holds by the same reason. Thus the two maps $g$ and $\bar f$ are related by the following diagram
$$\xymatrix{M\ar[r]^{q_M}\ar[d]_g\ar[rd]_f&
M\otimes_AR\ar[d]^{\bar f} \\
N\square_CB\ar[r]_{\varepsilon_B}&N.}$$
The adjunctions are obtained from this diagram for $f$ being the quotient map $q_M:M\to M\otimes_AR$, $M\in\mathcal M^B_A$ and the map $\sigma_B:  N\square_CB\to N$, $N\in\mathcal M^C$.
\end{proof}

The following proposition generalizes \cite[Theorem 1]{T1} see also \cite[Lemma 2.4.1]{Ps}.
\begin{proposition}\label{thm_T1} Let $A$ be a normal Hopf subalgebra of an $R$-flat Hopf algebra $B$, $C:=B/BA^+$. Assume that $B$ is faithfully flat as a left $A$-module. Then $A$ is a saturated Hopf subalgebra of $B$, $C$ is $R$-flat and the functors $\Phi$ and $\Psi$ establish an equivalence between categories $\mathcal{M}^B_A$ and $\mathcal{M}^{C}$.  Furthermore, $B$ is left faithfully co-flat over $C$.\end{proposition}
\begin{proof} The faithful flatness of $B$ over $A$ implies that $A$ is saturated in $B$. Indeed, tensoring the exact sequence
$$0\to A\to B\to B/A\to 0$$
on the right with $B$ over $A$ we obtain a split exact sequence (the splitting is given by the map $B\otimes_AB\to B$, $m\otimes n\mapsto mn$). By assumption $B$ is $R$-flat, hence so is  $B\otimes_AB$. Consequently $B/A\otimes_AB$ is $R$-flat. Now the faithful flatness of $B$ over $A$ implies that $B/A$ is $R$-flat, that is, $A$ is saturated in $B$ as an $R$-module.

    For each $M \in \mathcal{M}^B$, we have an  isomorphism
$$ \gamma_M: M \otimes B \cong M \otimes B,\quad \gamma_M (m \otimes b) = \sum m_{(0)}\otimes m_{(1)}b ,$$ 
with the inverse given by $m\otimes b \longmapsto \sum m_{(0)}\otimes S(m_{(1)})b$.

For $M \in \mathcal{M}^B_A$, consider $M\otimes A$ as a $B$-comodule by the diagonal coaction, we have the following diagram:
\begin{equation} \label{eq_gamma}\xymatrix{
M\otimes A\otimes B\ar[d]_{\gamma_{M\otimes A}}
\ar[rrr]^{r_M\otimes B-M\otimes l_B}&&&
M\otimes B\ar[d]_{\gamma_M}\ar[r]&M\otimes_AB
\ar[d]^{\gamma_{A,M}}\ar[r]&0\\
(M\otimes A)\otimes B\ar[rrr]_{r_M\otimes B-M\otimes\varepsilon_A\otimes B}&&&
M\otimes B\ar[r]_{q_M\otimes B\quad\  }
&(M\otimes_AR)\otimes B\ar[r]&0,}
\end{equation}
where, $r_M$, $l_B$ denote the actions of $A$ and $\varepsilon_A$ denotes the counit of $A$. Consequently we get an isomorphism 
$$\gamma_{A,M}: M \otimes_A B \xrightarrow\sim(M\otimes_AR) \otimes B=\Phi(M)\otimes B.$$  
In particular, for $M=B$ we have
\begin{equation}\label{eq_canonical}
\gamma_{A,B}:B \otimes_A B  \xrightarrow\sim  C \otimes B.\end{equation}
Now $B\otimes_AB$ is $R$-flat, that is $C\otimes B$ is $A$-flat. The faithful flatness of $B$ over $R$ implies that $C$ is $R$-flat.

Further, by the flatness of $B$ over $A$ and $R$, using \eqref{eq2refl}, we have, for $M \in \mathcal{M}^B_A$ and $N \in \mathcal{M}^{C}$,
\begin{equation}\label{eq.Psi} \Psi(N)\otimes_A B= (N\square_CB)\otimes_AB\stackrel{\eqref{eq2refl}}{=} 
N\square_C(B\otimes_AB)
\stackrel{\gamma_{A,B}}\cong N\square_C(C\otimes B)\cong
 N \otimes B,\end{equation}
 $$\sum_in_i\otimes b_i\otimes b\longmapsto \sum_i n_i\otimes \pi(b_{i(1)})\otimes b_{i(2)}b  \longmapsto \sum_i n_i\otimes b_ib.$$
  
Suppose that $\bar f: M\otimes_AR\longrightarrow N$ in 
$\mathcal{M}^{C}$ corresponds to $g: M 
\longrightarrow N\square_CB$ in 
$\mathcal{M}^B_A$. As shown in the proof of 
Lemma \ref{lem.adj1}, both morphisms come 
from a morphism $ f:M\to N$ in 
$\mathcal{M}^{C}$. Tensoring the diagram 
there with $B$  we get the following commutative diagram:
\begin{displaymath} \xymatrix{M \otimes B \ar[r]^{\gamma_M} 
\ar[d]_{g\otimes \id_B}&
M \otimes B \ar[r]^{q_M\otimes b} \ar[d]^{g\otimes \id_B}&
(M\otimes_AR)\otimes B \ar[d]^{{\bar f} \otimes \id_B }\\ (N\square_CB) \otimes B \ar[r]_{\gamma_{N\square_CB}}& (N\square_CB) \otimes B \ar[r]_{N\otimes\varepsilon_B\otimes B}&  N \otimes B.} \end{displaymath} By means of the right commutative diagram in \eqref{eq_gamma} we obtain the following commutative diagram:
\begin{displaymath}
\xymatrix{M \otimes_ A B \ar[r]^{\gamma_{A,M}} 
\ar[d]^{g\otimes \id_B}& 
(M\otimes_AR)\otimes B
\ar[d]^{{\bar f} \otimes \id_B }\\
(N\square_CB) \otimes_A B \ar[r]^{\simeq}& 
N \otimes B,} \end{displaymath}
where the lower horizontal map is nothing but the isomorphism in \eqref{eq.Psi}.
Since $B$ is faithfully flat over $A$ and $R$, $\bar f$ 
is  an isomorphism iff $g$ is.
Consequently, the adjunctions 
$M\longrightarrow \Psi\Phi(M)= 
(M\otimes_AR)\square_{C}B$,  
and $\Phi\Psi(N)= (N\square_{C}B)\otimes_AR \longrightarrow N$, are isomorphisms.
Thus $\Phi$ and $\Psi$ are equivalences.

Finally, the isomorphism in \eqref{eq.Psi} shows that the functor $(-\square_CB)\otimes_AB$ is fully faithful, as $B$ is faithfully flat as an $A$-module. Thus we conclude that $B$ is faithfully co-flat as a left $C$-comodule.
\end{proof}

\begin{remarks}\label{rmk1}
The experts may notice that the assumption of Proposition \ref{thm_T1} is too restrictive, compared with the original result of Takeuchi. In fact, it suffices to assume that ``$A$ is a right coideal subalgebra in $B$ and $B$ is faithfully flat as a left $A$-module", furthermore ``$R$ can be any commutative ring". The proof is the same and we leave it to the interested reader to carry out.
\end{remarks}

\subsection{Conormal quotient maps}
Consider now the dual situation.  Let $B\to C$, $b\mapsto \overline b$, be a conormal quotient map of $R$-flat Hopf algebras. Let
$$A:=B^\text{co$(C)$}.$$
If $T \in \mathcal{M}_A$, then $T \otimes_AB$ is in $\mathcal{M}^ {C}_B$, where the $C$-comodule structure is given by that on $B$:
$$t \otimes b \mapsto  \sum_{b} t \otimes b_1 \otimes \overline b_2,\quad \text{ for $b \in B, t \in T$}.$$ 
This yields a functor 
$-\otimes_A B:\mathcal{M}_A\longrightarrow \mathcal{M}^ {C}_B$.
\begin{lemma}  The functor 
$(-)^\text{co$(C)$}: \mathcal{M}^ {C}_B \longrightarrow    \mathcal{M}_A, Q \longmapsto Q^\text{co$(C)$} $
is right adjoint to $-\otimes_{A}B$:
$${\sf  Hom}_B^{C}(T \otimes _A B, Q)  \cong  {\sf Hom}_A(T, Q^\text{co$(C)$}).$$ 
The  adjunctions are given by
$$\eta: T \longrightarrow (T \otimes_AB)^\text{co$(C)$}, \quad t \mapsto t \otimes 1,$$ 
$$\zeta: Q^\text{co$(C)$}\otimes_AB\longrightarrow Q,\quad q \otimes b \mapsto qb.$$
\end{lemma}
\begin{proof} 
Let   $T \in \mathcal{M}_A$ and $Q \in \mathcal{M}^{{C}}_B$. Then $f: T \longrightarrow Q^\text{co$(C)$}$ in $\mathcal{M}_A$ corresponds to $g: T \otimes_AB \longrightarrow Q$ in $\mathcal{M}^B_A$ by means of the following commutative diagram:
\begin{equation}\label{eq3.5}\xymatrix{T\ar[d]_f\ar[r] & T\otimes_AB\ar[d]^g\\
Q^\text{co$(C)$}\ar@{^(->}[r]& Q,}\end{equation}
where the upper horizontal map is given explicitly by $t\mapsto t\otimes 1$.
In fact, the $C$-colinearity of $g$ forces the image of $f$ to be in $Q^\text{co$(C)$}$ and vice-versa.
 \end{proof}

The following proposition generalizes \cite[Theorem 2]{T1}.
\begin{proposition} \label{T_2}
Suppose that  $B$ is left faithfully co-flat over ${C}$.
Then the  above functors establish an equivalence between categories $\mathcal{M}_A$ and $ \mathcal{M}^ {{C}}_B$. Furthermore, $B$ is fully faithful as a left $A$-module.
\end{proposition}
\begin{proof}
For any $P\in\mathcal M^C_B$ we have an isomorphism
$$\theta_P:P\otimes B\xrightarrow{\sim} P\otimes B,\quad p\otimes b\mapsto \sum_b pb_1\otimes b_2,$$
with the inverse map given by $p\otimes b\mapsto \sum_b pS(b_1)\otimes b_2.$
Since $B$ is $R$-flat we have the following commutative diagram with exact rows:
$$\xymatrix{0\ar[r]&
Q^\text{co$(C)$}\otimes B\ar[d]_{\theta_{C,Q}}
\ar[r]&
Q\otimes B\ar[rr]^{(\rho_Q-Q\otimes u)\otimes B}\ar[d]^{\theta_Q}&&
(Q\otimes C)\otimes B\ar[d]^{\theta_{Q\otimes C}}\\
0\ar[r]&Q\square_CB\ar[r]& 
Q\otimes B\ar[rr]_{\rho_Q-\rho_B} &&
Q\otimes C\otimes B,}$$
where $u:R\to C$ denotes the unit map for $C$ 
and in the definition of 
$\theta_{Q\otimes C}$ the action of $B$ on 
$Q\otimes C$ is diagonal. This forces the first 
vertical map to be an isomorphism.
Thus we have an isomorphism
$$\theta_{C,Q}:Q^\text{co$(C)$}\otimes B\xrightarrow\sim Q\square_CB,\quad q\otimes b\mapsto \sum_b qb_1\otimes b_2,$$
for any $Q$ in $\mathcal M^C_B$. In particular we have
$$\theta_{C,B}:A\otimes B\xrightarrow\sim B\square_CB,\quad a\otimes b\mapsto \sum_bab_1\otimes b_2.$$
The inverse map is given by $b\otimes b'\mapsto \sum_{b'}bS(b'_1)\otimes b'_2.$

Tensoring \eqref{eq3.5} with $B$ and twisting it 
with the map $\theta$ we have a commutative diagram
$$\xymatrix{
T\otimes B\ar[d]_{f\otimes B}\ar[rr]&&
(T\otimes_AB)\otimes B\ar[d]^{g\otimes B}  \\
Q^\text{co$(C)$}\otimes B\ar@{^(->}[rr]&&Q\otimes B,}$$
where the upper and lower horizontal maps are given explicitly by $t\otimes b\mapsto\sum_b (t\otimes b_1)\otimes b_2$ and $q\otimes b\mapsto \sum_b qb_1\otimes b_2$, respectively.
 By means of the equality in \eqref{eq2}, we have
\begin{equation}\label{eq.T}
T\otimes B=T\otimes_A(A\otimes B)
\cong T\otimes_A(B\square_{C}B)=
(T\otimes_AB)\square_{C}B\subset (T\otimes_AB)\otimes B, \end{equation}
with the composed map given by $t\otimes b\mapsto \sum_b t\otimes b_1\otimes b_2$. Thus the above diagram reduces to the following commutative diagram with horizontal morphisms being isomorphisms:
$$\xymatrix{
T\otimes B\ar[d]_{f\otimes B}\ar[rr]&&
(T\otimes_AB)\square_CB
\ar[d]^{g\square_C B}  \\
Q^\text{co$(C)$}\otimes  B\ar[rr]_{\theta_{C,Q}}&&Q\square_C B.}$$
Consequently, $f$ is  an isomorphism iff $g$ is, as $B$ is faithfully co-flat over $C$.
This implies the adjunctions
 $T\longrightarrow
 (T\otimes_AB)^\text{co$ {C}$}$, $t \mapsto t \otimes 1$ ;
and $ Q^\text{co$(C)$}\otimes_AB \longrightarrow Q$, $ q\otimes b \mapsto qb$ are isomorphisms, where $t \in T, q\in Q$ and $b \in B$.
 Thus the functors are equivalences.
 
 Finally, the isomorphism in \eqref{eq.T} shows that the functor $(-\otimes_AB)\square_CB$ is fully faithful. As $B$ is faithfully co-flat as a left $C$-comodule, we conclude that $B$ is faithfully flat as a right $A$-module.
\end{proof}

\begin{remarks}\label{rmk2}
Again, the experts can see that Proposition \ref{T_2} can be formulated in a stronger form, namely it suffices to assume: ``$C$ is an $R$-flat Hopf algebra, $B$ is flat over $R$ and is faithfully co-flat as a left $C$-comodule,  where $R$ is any commutative ring''. Note that under this assumption, $A:=B^\text{co$(C)$}$ is faithfully flat over $R$. Indeed, the isomorphism $\theta_{C,B}:A\otimes B\xrightarrow\sim B\square_CB$ and the isomorphism in Subsection \ref{subsection2.2}~(for $A=R$ and remembering that $B$ is left faithfully co-flat over $C$) yields an isomorphism
$$(-\otimes A)\otimes B\cong-\otimes(B\square_CB)\cong (-\otimes B)\square_CB,$$
which shows that $A$ is faithfully flat over $R$.
\end{remarks}

\begin{proposition}
\label{f-co-flat}
Let $B$ be an $R$-flat Hopf algebra. Let $\mathcal{S}(B)$ be the set of all  normal Hopf subalgebras  of  $A$ such that  $B$ left faithfully flat over $A$ (in particular, $A$ is saturated in $B$) and let
$\mathcal{I}(B)$ be the set of  all  saturated normal Hopf  ideals $I$ such that $B$ is left faithfully co-flat over $B/I$. Then
 $$ \Phi:  \mathcal{S}(B) \longrightarrow  \mathcal{I}(B),\quad
 A\mapsto BA^+$$  
and  
$$ \Psi:  \mathcal{I}(B) \longrightarrow  \mathcal{S}(B),\quad
I\mapsto B^{co(B/I)}$$  
are mutually inverse maps.
\end{proposition}
\begin{proof}  
Let $A\in \mathcal S(B)$. Set $C:=B/BA^+=B\otimes_AR$. 
Then according to Proposition \ref{thm_T1}, $C$ is $R$-flat, the map $B\to C$ is conormal and $B$ is faithfully co-flat as a left $C$-comodule.
Moreover,  setting $N=R$ in \eqref{eq.Psi} and noticing that $A\subset B^\text{co$(C)$}$, we conclude that $A=B^\text{co$(C)$}$.

Conversely, assume that $I\in \mathcal I(B)$ and 
let $B\to C=B/I$ be the corresponding conormal 
quotient map of flat  Hopf algebras. Then  $A:=B^\text{co$(C)$}$ is a saturated normal 
Hopf subalgebra of $B$, by Lemma \ref{normal}~(ii). According to Proposition \ref{T_2}, $B$ is left faithfully flat over $A$. Moreover, setting 
$T=R$ in \eqref{eq.T} and noticing that the map 
$R\otimes_AB\to C$ is surjective as $A^+\subset I$, we conclude 
that $R\otimes_AB\cong C$.
\end{proof}

\begin{remarks}\label{rmk3}
The correspondence in Proposition \ref{f-co-flat} holds for any commutative ring $R$:
{\em There is a one-to-one correspondence between the set of all normal Hopf subalgebras $A$ of $B$ such that $B$ is left faithfully flat as an $A$-module (which implies that $B/A$ is $R$-flat) and the set of all normal Hopf ideals $I$, such that $C:=B/I$ is $R$-flat (i.e. $I$ is saturated in $B$ as an $R$-submodule) and $B$ is left faithfully co-flat over $C$.}

 Indeed, it is a consequence of Proposition \ref{thm_T1}, Remark \ref{rmk1} that the map $\Phi:A\mapsto I=BA^+$ is well-defined. For the inverse map we first set $C:=B/I$ and employ Remark \ref{rmk2} to see that $A:=B^\text{co$(C)$}$ is an algebra, which is faithfully flat over $R$  and $B$ is faithfully flat as a left $A$-module. Consequently $A$ is a saturated $R$-submodule of $B$ (that is $B/A$ is $R$-flat), see proof of \eqref{thm_T1}. Now we can imitate the original proof of Schneider \cite[Lem~1.3(i)]{Sch93} to show that $A$ is a Hopf subalgebra of $B$, which is normal by its construction.
\end{remarks}

\section{The faithful flatness over Hopf subalgebras}\label{Sect4}

In this section,  we want to give a faithful flatness criterion for a map of $R$-algebras by means of the flatness of its fibers at residues fields of $R$. Compared with the well-known flatness criterion of Grothendieck  \cite[11.3.11]{EGA}, our criterion does not need any finiteness conditions. 

Let $R$ be a Dedekind ring with fraction field $K$, a residue field will be denoted by $k$.
  The proof of the following lemma is similar to that of \cite[Thm 4.1.1]{DH}.
\begin{lemma}\label{flat0}
  Let $A$  be an $R$-algebra and $B$ be an $A$-module such that both $A, B$ are flat over $R$. Then  $B$ is flat over $A$
 if   and only if $B_k$ is flat over $A_k$,  for $k$ being  the fraction field  and any residue field of $R$.
\end{lemma}
\begin{proof}
 We prove the ``only if'' claim. Assume that $B$ is a left $A$-module and that $B_k$ are flat over $A_k$ for all $k$ as above. To show that $B$ 
is flat over $A$, it suffices to check that ${\rm Tor}^{A}_1(M,B)=0$ for any  right $A$-module 
$M$ (see \cite[Exer.~3.2.1, p.69]{Weibel}).
Choose  a left projective resolution $P_{*}$ of 
$B$ over $A$, \cite[Lemma 2..2.5]{Weibel}. Since $A$ is $R$-flat, so are the 
terms of $P^*$ and since a submodule of an 
$R$-flat module is again flat, the resolution 
$P_*\to B$ remains a resolution after base 
change. On the other hand, the projectivity is 
preserved by base change, therefore, for any 
residue field $k:=k_p$,   $p\in \spec(R)$,
$P_{*} \otimes k$ is a left projective resolution 
of $B \otimes k$ over $A \otimes k$. We have
$$ (M \otimes k)  \otimes_{(A \otimes k)}(P_{*} \otimes k) \cong (M \otimes_{A} P_{*}) \otimes k, $$ implying
$$H_i ((M \otimes_{A} P_{*}) \otimes k) \cong {\rm Tor}^{A \otimes k}_i (M \otimes k, B \otimes k),\mbox{ for all } i \ge 0.$$
First assume that $M$ is $R$-flat. Since $M \otimes_{A} P_{*}$ is $R$-flat, we can apply the universal  coefficient theorem, (see, e.g., \cite[Thm~3.6.1]{Weibel}). Thus, for each $i\geq 1$, we have an exact sequence
$$ 0\to H_i (M \otimes_{A} P_{*}) \otimes k \to H_i ((M \otimes_{A} P_{*}) \otimes k) \to {\rm Tor}^{R}_1({ H_{i-1}}(M \otimes_{A} P_{*}), k)\to 0.$$
That is, for all $i\geq 1$,
$$0\to {\rm Tor}^{A}_i (M, B) \otimes k \to {\rm Tor}^{A\otimes k}_i (M \otimes k, B \otimes k) \to {\rm Tor}^{R}_1({\rm Tor}^{A}_{i-1} (M, B) , k)\to 0.$$
By assumption
$A\otimes k \longrightarrow B\otimes k$ is  flat.  So
${\rm Tor}^{A \otimes k}_i (M \otimes k, B \otimes k) = 0$, for all $i\geq 1$. Consequently
$${\rm Tor}^{R}_1({\rm Tor}^{A}_{i-1} (M, B) , k) = {\rm Tor}^{A}_i (M, B) \otimes k  =0, \mbox{ for all } i\geq 1.$$
This holds for any residue field and the fraction field of $R$, hence
${\rm Tor}^{A}_i (M, B)$ is  flat over $R$ for $i\geq 0$  and we conclude that
${\rm Tor}^{A}_i (M, B) =0$ for all $i\geq 1$.

Let now $M$ be an arbitrary right $A$-module. Note that  $R$-torsion submodule $M_\tau$ of $M$ is also a  right $A$-submodule. Then the quotient module $M/M_\tau$ is $R$-flat and from the exact sequence
$${\rm Tor}^{A}_1(M_\tau,B)\to {\rm Tor}^{A}_1(M,B)\to {\rm Tor}^{A}_1(M/M_\tau,B)\to\ldots$$
it suffices to show ${\rm Tor}^{A}_1(M,B)=0$ for $M$ being $R$-torsion.

For each non-zero ideal $p\subset R$, the submodule 
$M_p$ of elements annihilated by elements of 
$p$, is also an $A$-submodule. As $M$ is 
torsion, it is the direct limit of $M_p$. Since the 
Tor-functor commutes with direct limits, one can 
replace $M$ by some $M_p$. Since  each 
non-zero ideal $p \subset R$ is a product of finitely 
many prime ideals, each $M_p$ has a filtration, 
each grade module of which is annihilated by a 
certain non-zero prime ideal. Thus using 
induction we can reduce to the case $M$ is 
annihilated by a prime ideal $p$. In this case 
$M=M\otimes k_p$ is an 
$A\otimes k_p$-module, where $k_p:=R/p$ and we have
$$M\otimes_{A}P_*= M\otimes_{A\otimes k_p}(P_*\otimes k_p).$$
 Since $P_*\otimes_Rk_p$ is an $A\otimes_Rk_p$-projective resolution of $B\otimes_Rk_p$, we see that
 $${\rm Tor}^{A}_i(M,B)={\rm Tor}^{A\otimes k_p}_i(M,B\otimes k_p)=0,$$
 as $B\otimes k_p$ is flat over $A\otimes k_p$.
\end{proof}

The next theorem is a generalization of the  faithful flatness theorem for flat commutative algebras over a Dedekind ring \cite{DH}\footnote{We thank the anonymous referee for his or her help that allows us to remove the finiteness assumption on $A$ in a previous version}.
\begin{proposition}\label{flatness}
  Let $A$  be an $R$-algebra and $B$ be an $A$-module such that both $A, B$ are flat over $R$.  Then   $B$ is faithfully flat over $A$
 if and only if $B_k$ is faithfully flat over $A_k$  for $k$ being either the fraction field or any residue field of $R$.
\end{proposition}
\begin{proof}
Assume that  $B$ is faithfully flat over $A$ as a left module.
Let $k$ be either the fraction field or  a residue field of $R$. Then for any  right $A_k$-module $M$ which satisfies
 $M \otimes_{A_k} B_k = 0$, we have $M \otimes_{A_k} B_k \cong M \otimes_A B =0$, hence $M=0$. Thus  $B_k$ is faithfully flat over $A_k$.

Conversely, assume $B_k$ is left faithfully flat over $A_k$. By the previous lemma, we know that $B$ is left flat over $A$. We show the faithfulness. Let $M$ be a right $A$-module  such that $M\otimes _{A}B=0$. Then we have
$$M_k\otimes _{A_k}B_k\cong (M\otimes _{A}B)\otimes _Rk=0,$$
 for  any residue field $k$ of $R$. 
Since  $B_k$ is (left) faithfully flat over $A_k$, we have $M_k=0$. The argument  is also valid for the fraction field $K$. This mean $M$ is $R$-torsion. 
If $M$ in non-zero then  $M$ contains non-zero elements annihilated by some maximal ideal $m$. Denote by  $M'$  the annihilator of   $m$ in $M$. Clearly, $M' \cong M'_k$  where $k=R/m$.  Since $B$ is flat over $A$,  the equality $M\otimes_AB=0$ implies that  $M'\otimes_{A}B= M'_k \otimes_AB=0$.  But then $M'_k=0$, as has been proved already. Therefore $M'=0$, a contradiction.
Thus $B$ is faithfully flat over $A$.
\end{proof}

As a corollary of Proposition \ref{flatness} and Nichols-Zoeller's theorem \cite{NZ89b}
we have  following
\begin{proposition}\label{flatness1}
Let $B$ be an $R$-finite flat Hopf  algebra and $A$ be a saturated Hopf  subalgebra of $B$.
Then $B$ is (left and right) faithfully flat over $A$.
 \end{proposition}
\begin{proof}
Since $B/A$  is $R$-flat we have $ A_k \subset B_k$ for any residue field and the fraction field $K$.   $B_k$ is free hence, in particular, faithfully flat over $A_k$, therefore Proposition \ref{flatness} ensures that $B$ is faithfully flat over $A$.
\end{proof}

It is known that every  finitely  presented  flat $A$-module  is  projective over $A$, (see \cite[Thm 3.2.7]{Weibel}). Hence we obtain immediately
\begin{corollary}\label{proj}
  Assume that $B$ is finite over $R$. Then  for any saturated Hopf subalgebra $A$, $B$ is (left  and right) projective as an $A$-module.
\end{corollary}

\begin{theorem}\label{flatness2}
Let $B$ be  a flat Hopf  algebra over Dedekind ring $R$ and $A$ be a saturated normal Hopf subalgebra of $B$. Assume that $A$ is $R$-finite.
Then $B$ is (left and right) faithfully flat as an $A$-module. Consequently $B$ is (left and right) faithfully co-flat over $C:=B/A^+B$ and we have equivalences 
$\mathcal{M}^{C}\cong \mathcal{M}^B_A$ and
$\mathcal{M}_A \cong \mathcal{M}^ {C}_B$.
 \end{theorem}
\begin{proof}
As $B/A$ flat over $R$, we have  $A_k$ is a normal  finite Hopf  subalgebra of  $B_k$ for $k$ being any residue field or the fraction field.  According to{\cite[Lemma 2.2 and  Theorem 2.4]{Sch93}},   $B_k$ is free over $A_k$.
The assertion follows by Proposition \ref{flatness}.
\end{proof}
\begin{question}
 Let $R$ be a DVR. Suppose that $B$ is $R$-projective and  $A$ is finite  normal  Hopf  subalgebra of $B$. Under which circumstance will $B$ be projective or free over $A$.
\end{question}
If $R$ is local and $B$ is $R$-finite free, then it is known that $B$ will be free over any $R$-projective Hopf subalgebra, see \cite[Remark 2.1]{Sch92}.

\section{Module of integrals on a Hopf algebra}
\label{SecInt}
We study in this section the module of integrals 
on an $R$-projective Hopf algebra, where $R$ is a Dedekind ring. We show that 
this module is either 0 or is projective of rank 1 
over $R$. We also prove some 
finiteness properties of Hopf algebras with non-zero integrals.
 \subsection{The module of integrals}\label{integral} Let $R$ be a  Dedekind ring. Let $H$ be an $R$-flat Hopf algebra. A linear map $\varphi:H\to R$ is called a left integral if it satisfies
$$\sum_hh_1\varphi(h_2)=\varphi(h),\quad \text{ for all } h\in H.$$
Similarly we have the notion of right integrals. Notice that $\varphi$ is a left (resp. right) integral on $H$ iff its is a homomorphism of $H$-comodules, where $H$ coacts on itself on the left (resp. the right). The left (resp. right) integrals on $H$ form an $R$-module, called module of left (resp. right) integrals, denoted by $I^l_H$ (resp. $I^r_H$).

\begin{lemma}\label{FiniteProj} For any $M \in \mathcal{M}^{H}$ which is $R$-finite projective,  ${\sf Hom}^H( H, M) \cong   M \otimes   {\sf Hom}^H( H, R)$.  
\end{lemma}
\begin{proof} Using the isomorphism of $H$-comodules 
$$N\otimes H\cong N\otimes H,\quad n\otimes h\mapsto \sum_n n_0\otimes n_1h,$$
where on the source $H$ coacts diagonally and 
on the target $H$ coacts through the coproduct 
on itself,  we have isomorphisms:  
$$M \otimes{\sf Hom}^H(H, R)\cong {\sf Hom}^H((M^*) \otimes H, R) \cong
{\sf Hom}^H (M^*\otimes H, R)\cong
{\sf Hom}^H(H, M),$$
which are explicitly given as follows:
\begin{eqnarray*}m\otimes \varphi&\mapsto&\varphi_m:\eta\otimes h\mapsto \eta(m)\varphi(h),\quad \eta\in M^*\\
&\mapsto& \widehat\varphi_m:\eta\otimes h\mapsto\sum_\eta\eta_0(m)\varphi(\eta_1h) =\sum_m\eta(m_0)\varphi(S(m_1)h)\\
&\mapsto&\overline\varphi_m:h\mapsto \sum_mm_0\varphi(S(m_1)h).\end{eqnarray*}
  \end{proof}

\begin{corollary}\label{summand}
Assume that $H$ has an $R$-finite direct summand as a right comodule on itself. Then $H$ possesses a non-zero left integral.
\end{corollary}

\begin{remark}
It is not known if the converse to the claim of Corollary \ref{summand} holds true.
\end{remark}

Let us consider the situation of  Theorem \ref{flatness2}.
\begin{theorem}\label{B-C.int}
Assume that $A$ is an $R$-finite saturated normal Hopf subalgebra of an $R$-flat Hopf algebra $B$. Let $C:=B\otimes_AR$ be the quotient Hopf algebra. Then $B$ possesses a non-zero left (resp. right) integral iff $C$ does.
\end{theorem}
\begin{proof} The idea is to use the equivalence between $\mathcal M^C$ and $\mathcal M^B_A$ established in Proposition \ref{thm_T1}, based on the fact that $B$ is faithfully flat on $A$ by means of Theorem \ref{flatness2}.

The submodule of right integral-elements in $A$ is defined as follows:
$$A_\varepsilon :=\{ t\in A |  
ta=\varepsilon(a)t, \forall a \in A \}.$$ 
It follows from the definition, that this submodule is saturated in $A$.
Since $A$ is $R$-finite, according to \cite{Par},  $A_\varepsilon$ is projective  of rank 
$1$ as an $R$-module, furthermore the  antipode of $A$ is bijective.

Assume that $B$ posseses a non-zero integral.
According to Lemma \ref{FiniteProj}, there is an isomorphism
$$A\otimes {\sf Hom}^B(B,R)\cong {\sf Hom}^B(B,A),\quad a\otimes \varphi\mapsto
\varphi_a:b\mapsto \sum_aa_1\varphi(S(a_2)b).$$
Notice that $\varphi$ satisfies
$$\sum_aa_1\varphi(S(a_2)b)=
\sum_b\varphi(S(a)b_1)b_2.$$ 
Let $t$ be such that $S(t)$ is a right integral-element in $A$ (i.e. $t$ is a left integral-element) and $\varphi$ be a right integral on $B$ (i.e. $\varphi\in {\sf Hom}^B(B,R)$). We show that $\varphi_t:B\to A$ is $A$-linear hence it will yield a right integral on $C$.
Indeed, this follows immediately from the definition of $t$, $\varphi$ and the normality of $A$ in $B$. One has to check that
$$\varphi_t(ba)=\varphi_t(b)a.$$ 
Recall that $A$ is normal in $B$, thus we have $\sum_bb_1aS(b_2)\in A$. Hence, using the definition of $t$ we have
$$\varphi_t(ba)=\sum_{t,a,b}\varphi(S(t)b_1a_1)b_2a_2=
\sum_{t,a,b}\varphi(S(t)[b_1a_1S(b_2)]b_3)b_4a_2=
\sum_{t,b}\varphi(S(t)b_1)b_2a=\varphi_t(b)a.
$$
Hence $\varphi_t\in{\sf Hom}^B_A(B,A)$. According to Lemma \ref{lem.adj1}, $\varphi_t$ induces a map $\psi:C\to R$ in $\mathcal M^C$, 
$$\psi(\bar b):=\varepsilon(\varphi_t(b))=\varphi(S(t)b),$$ where $\bar b\in C$ is the coset represented by $b\in B$. 

The converse statement is proved by following backward the arguments above.\end{proof}

\subsection{The comodule ${\sf Rat}(H^*)$}
In this subsection we shall use the comodule ${\sf Rat}(H^*)$ to study the module of integrals. We shall assume that $H$ is an $R$-projective  Hopf algebra with the $R$-dual denoted by $H^*$.  The convolution product on 
$H^*$ is distinguished by a dot ``$\cdot$''.
Any right comodule $M \in \mathcal{M}^H$ is a left 
module over $H^*$ by the action 
$$H^* \otimes M \to M,\quad f\cdot m=\sum_m m_0\otimes f(m_1),\quad
f\in H^*, m\in M.$$
One shows that $\mathcal{M}^H$ is a full 
subcategory of $_{H^*}\mathcal{M}$ of left $H^*$-modules. An $H^*$-module $M$ is called rational if the action of $H^*$ is induced from a coaction of $H$ in the way described above.
Each $M \in  { _{H^*}\!\mathcal{M}}$ contains a unique maximal rational submodule 
denoted by ${\sf  Rat}(M)$ - the sum of all rational submodules of $M$.
 The rational module ${\sf Rat}(H^*)$ can be characterized by  finiteness conditions (see \cite[5.3, 5.4]{Wis} or \cite[1.4]{CM}): 
\begin{eqnarray*}{\sf Rat}(H^*)&=&\{f \in H^*| H^*\cdot f \textit{ is an $R$-finite submodule of $H^*$}\}\\
&=&\{f \in H^*| f\cdot H\textit{ is an $R$-finite submodule of $H$}\}.\end{eqnarray*}
 
One checks that ${\sf Rat}(H^*)$ is a right Hopf module 
and hence  (\cite[Lemma 3.3]{CM})
\begin{equation}\label{Hop_module}{\sf Rat}(H^*)\cong{\sf Rat}
(H^*)^\text{co$(H)$}\otimes H.\end{equation}

Notice that by definition, ${\sf Rat}(H^*)^\text{co$(H)$}$ consists of linear maps $\varphi:H\to R$ satisfying
$$\rho(\varphi)=\varphi\otimes 1,$$
where $\rho$ denotes the coaction of $H$ on ${\sf Rat}(H^*)$. This is equivalent to  saying that $\varphi$ is a left integral on $H$.
 
\begin{lemma} The $R$-module ${\sf Rat}(H^*)$ is flat and it is saturated in $H^*$. 
\end{lemma}
\begin{proof} The first claim is obvious. We prove the second claim.
For any non-zero $r \in R$ and $f \in H^*$, suppose that $rf \in {\sf Rat}(H^*)$. Then $(rf)\cdot H\subset H$ is an $R$-finite submodule. Since $H$ is projective, the saturation of this submodule in $H$ is $R$-finite, thus $f\cdot H$ is also $R$-finite.
\end{proof}

\begin{proposition}\label{pro.IntMod}
Let $H$ be an $R$-projective Hopf algebra.
 Suppose $H$ possesses a non-zero integral. Then the module of  integrals on $H$ is
$R$-projective of rank one. Further, the module of integrals respects base change, i.e. $(I^l_H)_k\cong I^l_{H_k}$ for $k$ being any residue field and
$(I^l_{H})_\mathfrak{m}\cong I^l_{H_\mathfrak{m}}$
for any prime ideals $\mathfrak m$.
\end{proposition}
\begin{proof} Let $K$ be the quotients field of $R$. Any $R$-linear map $f:H\to R$ determines a $K$-linear map
$f_K:H_K\to K$. The induced map  $(H^*)_K \to (H_K)^*$ is injective (generally not surjective!), giving rise to the inclusion ${\sf Rat}(H^*)_K \inj {\sf Rat}((H_K)^*)$, and hence the inclusion 
$$({\sf Rat}(H^*)^\text{co$(H)$})_K\inj {\sf Rat}((H_K)^*)^\text{co$(H_K)$}.$$
Thus if $H$ possesses a non-zero left integral, i.e. $I^l_H\neq 0$, then its flatness implies $(I_H^l)_K\neq 0$ and hence has dimension 1 over $K$. Consequently, if $V$ is an non-zero $R$-submodule of $I^l_H$ the quotient $I^l_H/V$ has to be $R$-torsion. On the other hand, there exists $a\in H$ such that the $R$-linear map $\varphi:I^l_H\to R$, $f\mapsto  f(a)$ is non-trivial. Then $I^l_H/\ker\varphi$ is $R$-torsion free. By the remark above $\ker\varphi=0$, i.e. $\varphi$ is injective and $I^l_H$ is isomorphic to an ideal of the Dedekind ring $R$, hence is a projective module of rank 1 over $R$.
In particular we have $(I^l_H)_k$ is non-zero for any residue field $k$\footnote{ This simple argument for the projectivity and the rank of $I_H^l$ was suggested to us by the anonymous referee.}.

Let $k$ be a residue field of $R$, $k=R/\mathfrak m$. Since $H$ is $R$-projective and $k$ is $R$-finite, we have
$${\rm Hom}_R(H,R)\otimes_Rk\cong {\rm Hom}_R(H,k)={\rm Hom}_k(H_k,k).$$
Thus $(H^*)_k\cong (H_k)^*.$ 
 The above lemma shows that the natural map
$$\xymatrix{
{\sf Rat}(H^*)_k\ar@{^(->}[r]& (H^*)_k\cong (H_k)^*}$$
is injective, its image is by construction contained in ${\sf Rat}((H_k)^*)$. This map also induces the inclusion on the co-invariants: 
$$(I_H^l)_k \to I_{H_k}^l.$$
Thus, if $ (I_H^l)_k\neq 0$, it has to be one-dimensional and the map above is an isomorphism. The same holds for the localization at any prime ideals of $R$, as this functor is exact, and faithful when restricted to flat modules.
\end{proof}
\begin{corollary}\label{S.inj}
Assume that an $R$-projective Hopf algebra $H$ possesses a non-zero left integral. Then the antipode of $H$ is bijective.
\end{corollary}
\begin{proof}
It is a classical result due to R. Radford that over a field, the antipode of a Hopf algebra with a non-zero integral is bijective, (\cite[Section 3, Proposition 2]{Ra}). Thus the antipode of $H$ has all fibers being bijective, hence it is itself bijective.
\end{proof}
\begin{corollary}\label{enough_integral}
Assume that an $R$-projective Hopf algebra $H$ possesses a non-zero left integral. Then $H$ ``has enough integrals'' in the sense that the map
\begin{equation}\label{eq.Sch} I^l_H\otimes H\to R,\quad \varphi\otimes h\mapsto \varphi(h)\end{equation}
is surjective, cf. \cite[Rmk~3.2~(4)]{Ps1}.
\end{corollary}
\begin{proof}
If $R$ is a discrete valuation ring (or more general, a PID), the conclusion if obvious. For a general Dedekind ring $R$, the localization of $R$ at a maximal ideal $\mathfrak m$ is a discrete valuation ring. By means of Propostion \ref{pro.IntMod}, all localizations at maximal ideals of the map $I^l_H\otimes H\to R$ are surjective, hence so is the original map.
\end{proof}

Corollary \ref{enough_integral} together with a remark of Schneider-Schauenburg \cite[Rmk~3.2~(4)]{Ps1} yields the following result.
 
\begin{proposition}\label{Integral2}Let $H$ be  an $R$-projective Hopf algebra. Then  $H$ possesses a non-zero left integral if and only if $H$ is projective in ${}_{H^*}\!\mathcal{M}$. In particular, if $H$ possesses a non-zero left integral then it is projective in $\mathcal M^H$.
\end{proposition}

\begin{remark}
If $H$ possesses a non-zero integral, then so does $H_K$. The converse is not true as can be seen from the following example. Let $R$ be a discrete valuation ring with unifomizer $\pi$. Consider the commutative Hopf algebra
$$H:=R[x,y]/(x+y+\pi xy).$$ 
The Hopf algebra structure is give by \begin{eqnarray*}\Delta(x)=1\otimes x+x\otimes 1+\pi x\otimes x,\\ \Delta(y)=1\otimes y+y\otimes 1+\pi y\otimes y,\\ 
S(x)=y,\quad S(y)=x,\\ \varepsilon(x)=\varepsilon(y)=0.\end{eqnarray*}

In fact $H$ is the coordinate ring of the Neron blow-up of the group scheme $\mathbb G_{m,R}={\sf Spec}R[uv]/(uv-1)$ at the unit element of the closed fiber, see, e.g., \cite{DHS}.

Now we see that $H_K$ is isomorphic to $K[u,v]/(uv-1)$, by means of the map $u\mapsto 1+\pi x$, $v\mapsto 1+\pi y$. Thus $H_K$ possesses a non-zero integral. On the other hand, $H_k$ is isomorphic to $k[x]$ (the coordinate ring of the additive group), on which there is no non-zero integrals. Finally $H$ is projective   according to \cite[Prop.~3.1.7]{DH}.
\end{remark}

\begin{question}
Let $\mathcal H$ be a Hopf algebra over $K$, which possesses a non-zero left integral. Does there exist an $R$-projective Hopf algebra $H$ possessing a non-zero left integral such that $H_K\cong \mathcal H$? (For commutative Hopf algebras, such an $H$ is called a model of $\mathcal H$.)
\end{question}

\section{Projectivity over Hopf subalgebras} \label{Secpro}
Our aim in this section is to give a condition for an 
$R$-flat Hopf algebra to be projective over a saturated 
finite normal Hopf subalgebra. We shall keep 
the settings and assumptions of the previous section. 
Thus let $R$ be a Dedekind ring and assume that $B$ is an $R$-flat Hopf algebra and $A$ is an 
$R$-finite saturated normal Hopf subalgebra of
$B$. Then according to Theorem \ref{flatness2}, $B$ is faithfully flat over $A$ and faithfully co-flat over $C:=B\otimes_AR=B/A^+B$ (note that $A^+B= BA^+$), this last Hopf algebra is also $R$-flat. 

\subsection{} 
For $b\in B$ denote $\bar b$ its coset in $C$. The map 
$$ f_B:A_\varepsilon \otimes C
\longrightarrow B,\quad f_B(t\otimes\bar b):= tb$$ 
is well-defined.   
 The following result is a generalization of  \cite[Lemma 3.2]{L.Ka}.
\begin{proposition} \label{lm} The map $f_B$ is right $B$-linear, right $C$-colinear and bijective on its image. 
Thus $A_\varepsilon\otimes C \cong A_\varepsilon B$ as $B$-modules.
 \end{proposition}
\begin{proof}  It follows from definition that $f_B$ maps $A_\varepsilon$ surjectively onto $A_\varepsilon B$ and is right $B$-linear. To see that it is right $C$-colinear we notice that for $a\in A$, $\bar a=\varepsilon(a)\cdot 1\in C$. Hence $$\sum_tt_1\otimes\overline{t_2}=t\otimes 1.$$
Consequently $f_B$ is right $C$-colinear.

It remains to show that  $f_B$ is injective. Since $f_B$ is injective if its generic fiber $(f_B)_K=f_{B_K}$ is injective, and since the submodule of right integral-elements respect base change, it suffices to work with $B_K$ in stead of $B$. That mean we can work with Hopf algebras over a field. Consider the exact sequence $0\to A^+\to A \xrightarrow{t\cdot -} A$, where $t\cdot -$ is the left multiplication by $t$, where $t$ is a non-zero right integral-element. Tensoring this exact sequence on the right with $B$ over $A$, we obtain, by the flatness of $B$ over $A$, an injective map
$B/A^+B\to B$, which is exactly $f_B(t,-)$.
\end{proof}  

\begin{corollary}
If $B$ is $R$-projective then $C=B\otimes_AR$ is also $R$-projective.
\end{corollary}
\begin{proof}
It is well-known that over a Dedekind ring,  a submodule of a projective module is again projective, see, e.g.  \cite[Thms.~5.5.1, ~8.3.4]{Haz}. Thus $A_\epsilon\otimes C$ is $R$-projective. But $A_\epsilon$ is $R$-projective of rank 1, i.e. invertible as an $R$-module, hence $C$ is $R$-projective.
\end{proof}
 \begin{corollary}\label{B-C.proj}
Let $A$ be an $R$-finite saturated normal Hopf subalgebra of an $R$-projective Hopf algebra $B$. Assume that $B$ has a non-zero left integral. Then $C:=B/A^+B$ is projective in $\mathcal M^C$.
\end{corollary}
\begin{proof}
This is an immediate consequence of Theorem \ref{B-C.int}  and Proposition \ref{Integral2} and the previous corollary.\end{proof}
 
\subsection{Projectivity over normal Hopf subalgebra} 
 
The aim of this last section is to prove the projectivity of $B$ as a left (or right) $A$-module.
We need some properties of $(B,C)$-Hopf modules  (see  \cite[1.2, 1.4]{M}).
\begin{remarks} \label{rmk5.3}
\begin{enumerate} 
\item For any  right $B$-module $M$, $M \otimes C$ is in $\mathcal{M}^{C}_B$ by letting $C$ coact on itself and $B$ act diagonally:
$$ \rho:  M \otimes C \to (M \otimes C) \otimes C,\quad m \otimes c \mapsto\sum_c m \otimes c_1\otimes c_2,$$
$$ \mu:  (M \otimes C) \otimes B \to M \otimes C,\quad (m \otimes c)\otimes b\mapsto
 \sum_b mb_1 \otimes cb_2.$$ 
\item The module $C \otimes B$ is in $\mathcal{M}^{C}_B$ with $B$ acting on itself and $C$ coacting diagonally:
$$ \mu: (C \otimes B) \otimes B  \longrightarrow C \otimes B,\quad c \otimes b\otimes b'\mapsto c\otimes bb' ;$$
$$ \rho:(C \otimes B) \longrightarrow (C \otimes B) \otimes C, \quad c \otimes b \mapsto \sum_{c, b} c_1 \otimes b_1 \otimes c_2b_2.$$
\item If the antipode $S$ of  $B$  has an inverse   then 
 $$C \otimes B \longrightarrow B \otimes C, \quad c \otimes b \mapsto \sum_{b}  b_1 \otimes  c b_2$$
is an isomorphism in $\mathcal{M}_B^{C}$
with inverse map 
$$ B \otimes C \longrightarrow C \otimes B,\quad b \otimes c \mapsto \sum_{b} cS^{-1}(b_2) \otimes b_1.$$
\end{enumerate}
\end{remarks} 
\begin{lemma} \label{proj1}
 Assume that the antipode of $B$ is bijective and  $C$ is projective in $\mathcal{M}^{C}$. Then  $B \otimes C$ considered as an object of $\mathcal{M}_B^{C}$ as in \ref{rmk5.3}~(i) is  projective.
\end{lemma}
\begin{proof} Consider $C\otimes B$ as an object in $\mathcal M^C_B$ as in \ref{rmk5.3}~(ii) above, we have an isomorphism
   $$ {\sf Hom}_B^{C}({C} \otimes B, M)  \longrightarrow {\sf Hom}^{C}({C}, M), 
   \quad g \mapsto  g(- \otimes B),$$
with the inverse given by $f  \mapsto r_M  \circ ( f \otimes B)$.
Now by the isomorphism in \ref{rmk5.3}~(iii), we have
$${\sf  Hom}_B^{C}(B \otimes C, M) \cong {\sf  Hom}_B^{C}(C \otimes  B, M) \cong {\sf Hom}^{C}(C, M),$$ for any $M  \in  {\mathcal{M}}_B^{C}$.
\end{proof}
 
\begin{lemma}\label{proj2} Assume that  $B\otimes C$, considered as an object of $\mathcal{M}_B^{C}$ as in \ref{rmk5.3}~(i), is  projective. Then  $B$ is right   projective over $A$. 
 \end{lemma}
\begin{proof}  

Since $ \mathcal{M}_B^ {C}  \cong {\mathcal{M}}_A$ by means of
the functor  $(-)^\text{co$(C)$}:  \mathcal{M}_B^ {C} \longrightarrow {\mathcal{M}}_A $, we see that $(B\otimes C)^\text{co$(C)$}$ is projective in $\mathcal M_A$. Further, we have 
$$(B\otimes C )^\text{co$(C)$}\cong (B \otimes C)\square_{C} R\cong B \otimes (C\square_{C} R) \cong B,$$
where $A$ acts on the rightmost term by the right action. Hence $B$ is projective as a right $A$-module. 
\end{proof}
The above lemmas bring together:
\begin{theorem}\label{thm.proj}
Let $B$ be an $R$-projective Hopf algebra with  a non-zero left integral. Let $A$ be an $R$-finite saturated normal Hopf subalgebra of $B$.  Then $B$ is right  projective over $A$.  
\end{theorem} 
\begin{proof}
According to \ref{Integral2},  \ref{B-C.proj}, $C$ is projective in $\mathcal M^C$.
According to Lemma \ref{S.inj}, the antipode of $B$ is bijective. Hence we can apply Lemmas \ref{proj1}, \ref{proj2} to conclude that $B$ is right projective over $A$. 
\end{proof}
\section*{Acknowledgment} 
We wish to express our deep thanks  to the anonymous referee for  his or her  valuable comments and for pointing out our mistakes in previous versions. 
 
 This research is funded by Vietnam National Foundation for  Science and Technology Development (NAFOSTED) under the grant number 101.01-2011.34. The research of the first and second named authors is partially supported by the project Implementation of the Agreement “International Associated Laboratory FORMATH VIETNAM” between VAST and CNRS, Grant number VAST.HTQT.Phap.03/16-17.

\end{document}